\documentclass[12pt,a4paper]{amsart}
\usepackage{amsaddr}
\usepackage{amsmath,amscd}
\usepackage{latexsym,amsmath,amssymb,mathrsfs,setspace,enumerate,tikz}
\usepackage{graphicx}
\usepackage{subcaption}
\usepackage[T1]{fontenc}
\usepackage[utf8]{inputenc}
\usepackage{lmodern}
\usepackage{amssymb}
\usepackage{multicol}
\usepackage{tikz-network}
\usetikzlibrary{automata,positioning}
\usepackage[english]{babel} 
\usepackage{blindtext}

\theoremstyle{plain}
\newtheorem{theorem}{Theorem}[section]
\newtheorem{lemma}[theorem]{Lemma}
\newtheorem{proposition}[theorem]{Proposition}
\newtheorem{corollary}[theorem]{Corollary}
\theoremstyle{definition}
\newtheorem{definition}[theorem]{Definition}

\newtheorem{remark}[theorem]{Remark}
\newtheorem{example}[theorem]{Example}

\begin{document}
	
	\title{ On bi-ideals of ordered full transformation semigroups}
	
			\author[Minnumol P. K. and P. G. Romeo 	]{Minnumol P. K. and P. G. Romeo}
			
	\address{Department of Mathematics\\
		Cochin University of Science and Technology(CUSAT)\\Kochi, Kerala, India, 682022}
	\email{$pkminnumol@gmail.com,\, romeo_-parackal@yahoo.com$}
	
	\thanks{First author wishes to thank Cochin University of Science And Technology for providing financial support.\\
		Second author wishes to thank Council of Scientific and Industrial Research(CSIR) INDIA, for providing financial support.
		 }
	\subjclass{20M17, 20M12, 06F05}
	\keywords{ordered semigroup, bi-ideals, natural partial order, full transformation semigroup, Green's relations.}
	
	\begin{abstract}
	In this paper we describe the Green's relations on the semigroup of bi-ideals of ordered full transformation semigroup $ \mathcal{B}(\mathcal{T}_{X}) $ in terms of Green's relations  of ordered full transformation semigroup $ \mathcal{T}_{X} $.
	\end{abstract}
\maketitle
		
The bi-ideals of semigroups were introduced by R. A. Good and D.R. Hughes in 1952. It is a special case of the (m,n)-ideal introduced by S. Lajos who gave a characterization of semigroup by the set of their bi-ideals ( see cf.\cite{good},  \cite{mnideal},  \cite{lajos}). Further N. Kehayopulu studied the concept of bi-ideals in the case of ordered semigroups in \cite{kehayo biideal} and  \cite{mallick} S. Mallick and K. Hansda introduced a semigroup $\mathcal{B}(S)$ of all bi-ideals of an ordered semigroup $ S $ and they gave different charecterizations of $\mathcal{B}(S)$ for different regular subclasses of $ S $. In this paper we describe the Green's relations on the semigroup of bi-ideals of an ordered full transformation semigroup.

\section{Preliminaries}
First we recall definitions and results needed in the sequel.\\
 An \textit{ordered semigroup} 
$ (S, \cdot , \leq) $ is a poset $ (S, \leq) $ as well as a semigroup $ (S, \cdot) $ such that for any $ a, b, x \in {S},\, a \leq b $ implies $ xa \leq xb $ and $ ax \leq bx $. For $ A, B \subseteq S , \quad  AB = \{ab\, |\, a \in{A} \, \text{and} \,b \in{B}\} $ and 
$$ (A]  = \{x \in{S}\,|\, ( \exists \, a \in{A})\,\,\text{and}\, \,x \leq a\}$$
 is called \textit{downward closure} of $ A $.

\begin{definition}
Let $ S$ be an ordered semigroups. A nonempty subset $ A $ of $ S $ is called a \textit{left (right) ideal } of $ S $ if $ SA \subseteq A \,( AS \subseteq A) $ and $ (A]  = A $ and $A$ is called an \textit{ideal} of $ S $ if it is both a left and right ideal of $ S $.	
\end{definition}
	
\begin{definition} (see \cite{kehayo poni})	A nonempty subset $ A $ of an ordered semigroup $ S $ is called a 
	\textit{ bi-ideal } of $ S $ if $ ASA \subseteq A $ and $ (A]  = A $. A bi-ideal $ A $ of $ S $ is called a \textit{subidempotent bi-ideal } if $ A^{2} \subseteq A. $ 
\end{definition}
\begin{definition}
An ordered semigroup $ S $ is called \textit{regular} if for every $ a \in{S} $	there exists $ x \in{S} $	such that $ a \leq axa $ or equivalenty $ a \in (aSa] $. 
	An ordered semigroup $ S $ is called \textit{intra-regular} if for every $ a \in{S} $	there exists $ x ,y\in{S} $	such that $ a \leq xa^{2}y $ or equivalenty $ a \in (Sa^{2}S] $.	
\end{definition}
If $ S $ is an ordered regular semigroup then the bi-ideals and the subidempotent bi-ideals are the  same.

\begin{definition} (see \cite{mallick}) Let $ (S, \cdot , \leq) $ be an ordered semigroup and  $\mathcal{B}(S)$ be the collection all bi-ideals of $ S $. Then $\mathcal{B}(S)$ together with the the binary operation defined by 
	$$ A * B = (AB] $$ 
	is a semigroup and is called the semigroup of bi-ideals of the ordered semigroup $ (S, \cdot , \leq) $.
\end{definition}
The semigroup of bi-ideals $\mathcal{B}(S)$ of an ordered semigroup $ S $ is significant in the study of structure of the semigroup as is evident from the following theorems.
 \begin{theorem}\cite{mallick}
 	Let $ S $ be an ordered semigroup. Then $ S $ is regular if and only if the semigroup $\mathcal{B}(S)$ of all bi-ideals is regular.
 \end{theorem}
 
\begin{theorem}{\cite{mallick}}{\label{b}}
	An ordered semigroup $ S $ is both regular and intra-regular if and only if $ \mathcal{B}(S) $ is a band.
\end{theorem}

The principal left ideal, right ideal, ideal and bi-ideal generated by $ a \in{S} $ are denoted by $ L(a), R(a), I(a), B(a) $ respectively and defined as follows:
\begin{center}
	$ L(a) = (a \cup Sa] $\\
	$ R(a) = (a \cup aS] $\\
	$ I(a) = (a \cup Sa \cup aS \cup SaS] $\\
	$ 	B(a) = (a \cup aSa] $ 
\end{center}
The \textit{Green's relations}  $ \mathscr{L} ,\mathscr{R}, \mathscr{J}, \mathscr{H}, \mathscr{D}  $ on an ordered semigroup $ S $ are given as follows (see\cite{kehayo greens})
\begin{center}
	$ a\,\mathscr{L} \,b $ if and only if $ L(a) = L(b) $\\
	$ a\,\mathscr{R} \,b $ if and only if $ R(a) = R(b) $\\
	$ a\,\mathscr{J} \,b $ if and only if $ I(a) = I(b) $\\
	$ \mathscr{H} = \mathscr{L} \cap \mathscr{R} $\\
	$ \mathscr{D} = \mathscr{L} \vee \mathscr{R} $
\end{center}

\subsection{Green's relations on an ordered regular semigroup}

In a regular ordered semigroup $ S $ we have a particularly useful way of looking at the equivalences $ \mathscr{L} $ and $ \mathscr{R} $. If $ S $ is an ordered regular semigroup, for each $ a \in{S} $ there exist $ x \in{S} $ such that $ a \leq axa$, hence $ a \in (Sa] $ and 
$ a \in (aS] $ and so the Green's equivalence for  ordered regular semigroups are:
\begin{center}
	$ a\,\mathscr{L} \,b $ if and only if $ (Sa]  = (Sb] $\\
	$ a\,\mathscr{R} \,b $ if and only if $ (aS]  = (bS] $\\

\end{center}
\begin{proposition}{\label{1}}
	Let $ a $ and $ b $ be elements of an ordered regular semigroup $ S $. Then 	$ a\,\mathscr{L} \,b $ if and only if there exists $ x,y \in{S} $ such that $ a \leq xb $ and $ b \leq ya $. Also $ a\,\mathscr{R} \,b $ if and only if there exists $ u,v \in{S} $ such that $ a \leq ub $ and $ b \leq va $.
\end{proposition}
\begin{proof}
	Suppose that $ a\,\mathscr{L} \,b $ then by definition $ (Sa]  = (Sb] $. Clearly $ a \in (Sb] $ and $ b \in (Sa] $, which implies that $ a \leq xb $ and $ b \leq ya $ for some $ x, y \in{S} $.\\
	Conversely suppose that there exists $ x,y \in{S} $ such that $ a \leq xb $ and $ b \leq ya $. Let $ t \in (Sa] $, then $ t\leq sa $ for some $ s\in{S} $ and $ t \leq sxb \,\, \Rightarrow\,\, t \in (Sb]  $. Now let $ t' \in (Sb] $, then $ t'\leq s'b $ for some $ s'\in{S} $ and $ t \leq s'ya \,\, \Rightarrow\,\, t \in (Sa]  $. Hence $ (Sa]  = (Sb] $. So $ a\,\mathscr{L} \,b $. Similarly we can prove that $ a\,\mathscr{R} \,b $ if and only if there exists $ u,v \in{S} $ such that $ a \leq ub $ and $ b \leq va $.
\end{proof}

\section{Equivalence relation on semigroup of bi-ideals-$\mathcal{B}(S)$}

Let $ (S, \cdot , \leq) $ be an ordered semigroup and $\mathcal{B}(S)$ be the semigroup of  bi-ideals of $ S $. Next we proceed to define certain relations on the semigroup of bi-ideals of the ordered semigroup $S$ by making use of the Green's relations on $ S $ which turns out to be equivalence relations on $\mathcal{B}(S)$ and we call them as Green's relations on  $\mathcal{B}(S)$.

\begin{proposition}{\label{a}}
Let $ S $ be an ordered  semigroup and $\mathcal{B}(S)$ be the semigroup of bi-ideals of $ S $. For $ A,B \in \mathcal{B}(S) $, define the relation $\mathscr{L'}$  on $\mathcal{B}(S)$ by 
$ A \, \mathscr{L'} \, B \, \text {if and only if for each}\,  a \in{A}\,\text{there exists some }\, \, b \in{B} \,\text{such that}\,\,a\, \mathscr{L}_{S} \, b$  and vice-versa, 
 where $ \mathscr{L}_{S} $ is the Green's $ \mathscr{L} $ relation on $ S $. Then $\mathscr{L'} $ is an equivalence relation on $\mathcal{B}(S)$.
\end{proposition} 

\begin{proof} Let $ A \, \mathscr{L'} \, B $. Then for each $ a \in{A} $ there
exists $ b \in{B} $ such that $ a\, \mathscr{L}_{S} \, b  $ and for each $ b' \in{B} $ there exists $ a' \in{A} $ such that $ a'\, \mathscr{L}_{S} \, b'  $. 
Consider $ A, B, C \in \mathcal{B}(S) $, clearly $ A \, \mathscr{L'} \, A $ and when $ A \, \mathscr{L'} \, B$ then $B \, \mathscr{L'} \, A $ ie., $\mathscr{L'}$ is both  reflexive and symmetric. Again for 
 $ A \, \mathscr{L'} \, B $ and $ B \, \mathscr{L'} \, C $, for each $ a \in{A} $ there is a $ b \in{B} $ such that $   a\, \mathscr{L}_{S} \, b  $ and for each $ b \in{B} $ there exists $ c \in{C} $ such that $   b\, \mathscr{L}_{S} \, c  $. By the transitivity of $ \mathscr{L}_{S} $ we have $   a\, \mathscr{L}_{S} \, c  $. i.e. for each $ a \in{A} $ there exists $ c \in{C} $ such that $   a\, \mathscr{L}_{S} \, c  $. Similarly it is seen that for each $ c' \in{C} $ there exists $ a' \in{A} $ such that $   a'\, \mathscr{L}_{S} \, c'  $. Hence $ A \, \mathscr{L'} \, C $ and $\mathscr{L'} $ is an equivalence relation on $\mathcal{B}(S)$.
\end{proof}

\begin{proposition}{\label{2}}
	Let $ S $ be an ordered regular semigroup and $\mathcal{B}(S)$ the semigroup of bi-ideals of $ S $. Let $\mathscr{L}_{\mathcal{B}(S)} $ is the Green's $\mathscr{L} $- relation on $\mathcal{B}(S)$ and $\mathscr{L'} $ is the equivalence relation  defined in Proposition \ref{a}. Then $ \mathscr{L}_{\mathcal{B}(S)} \subseteq \mathscr{L'} $.
\end{proposition}
\begin{proof}
	For $ A, B $ be any two bi-ideals of $ S $ and let $ (A,B) \in \mathscr{L}_{\mathcal{B}(S)}  $. Then by the equivalent condition of the Green's $\mathscr{L} $- relation, there exists $ X,Y \in {\mathcal{B}(S)} $ such that $ A = X * B$ and $ B = Y * A $. i.e $ A = (XB]  $ and $ B = (YA] $.\\
	 If possible $ (A,B) \notin \mathscr{L'} $. Then either there exists $ a \in A $ such that $ (a,b) \notin \mathscr{L}_{S} \, \forall \, b \in{B}  $, or there exists $ b \in{B} $ such that $ (a,b) \notin \mathscr{L}_{S} \, \forall \, a \in{A}  $. Suppose that there is an  $ a \in A $ such that $ (a,b) \notin \mathscr{L}_{S} \, \forall \, b \in{B}  $, then for any $ b \in B $, there doesnot exists $ x $ or $ y \in S $ such that $ a \leq xb $ and $ b \leq ya $, which is a contradiction to the existence of either $ X $ or $ Y $. Hence $ (A,B)  \in \mathscr{L'} $ and so $ \mathscr{L}_{\mathcal{B}(S)} \subseteq \mathscr{L'} $.
\end{proof}
Similarly the relation $ \mathscr{R'} $ can be defined and can obtain the following propositions.

 \begin{proposition}{\label{r}}
	Let $ S $ be an ordered regular semigroup and $\mathcal{B}(S)$ be the semigroup of bi-ideals of $ S $. Suppose that $ A,B \in \mathcal{B}(S) $. Define a relation $\mathscr{R'}$  on $\mathcal{B}(S)$ as $ A \, \mathscr{R'} \, B $ if and only if for each $ a \in{A} $ there exists $ b \in{B} $ such that $ a\, \mathscr{R}_{S} \, b  $ and vice-versa, where $ \mathscr{R}_{S} $ is the Green's $ \mathscr{R} $ relation on $ S $. Then $\mathscr{R'} $ is an equivalence relation on $\mathcal{B}(S)$.
\end{proposition}

\begin{proposition}
	Let $ S $ be an ordered regular semigroup and $\mathcal{B}(S)$ be the semigroup of bi-ideals of $ S $. Suppose that $\mathscr{R}_{\mathcal{B}(S)} $ is the Green's $\mathscr{R} $- relation on $\mathcal{B}(S)$ and $\mathscr{R'} $ is the equivalence relation  defined by Proposition \ref{r}. Then $ \mathscr{R}_{\mathcal{B}(S)} \subseteq \mathscr{R'} $.
\end{proposition}

\begin{remark}
	The Prosition \ref{2} holds even if we drop the ordered regularity. For, let $ S $ be an ordered  semigroup and $\mathcal{B}(S)$ be the semigroup of bi-ideals of $ S $. Suppose that $\mathscr{L}_{\mathcal{B}(S)} $ is the Green's $\mathscr{L} $- relation on $\mathcal{B}(S)$ and $\mathscr{L'} $ is the equivalence relation  defined on $\mathcal{B}(S)$ as $ A \, \mathscr{L'} \, B $ if and only if for each $ a \in{A} $ there exists $ b \in{B} $ such that $ a\, \mathscr{L}_{S} \, b  $ and vice-versa where $ \mathscr{L}_{S} $ is the Green's $ \mathscr{L} $ relation on $ S $ and $ A,B \in \mathcal{B}(S) $. Then $ \mathscr{L}_{\mathcal{B}(S)} \subseteq \mathscr{L'} $.\\
	However the reverse inclusion $ \mathscr{L'}  \subseteq \mathscr{L}_{\mathcal{B}(S)} $ will not in general holds. For, consider the following example.
\end{remark}

\begin{example}
	The set $ S = \{a,b,c,d,e\} $ together with multiplication table:
	
	 \begin{center}
		\begin{tabular}{c|c|c|c|r|r}
			$.$& $ a $ & $ b $ & $ c $ & $ d $  & $ e $\\
			\hline
		$ a  $& $ a $ & $ b $ & $ c $ & $ d $  & $ e $\\
			\hline
				$ b $& $ a $ & $ b $ & $ c $ & $ d $  & $ e $\\
			\hline
				$ c  $& $ c $ & $ c $ & $ c $ & $ c $  & $ c $\\
			\hline
		$ d  $& $ d $ & $ d $ & $ d $ & $ d $  & $ d $\\	
			\hline
			$ e  $& $ c $ & $ c $ & $ c $ & $ c $  & $ c $\\	
		\end{tabular} 
	\end{center}
and patial order:\\
$ \leq := \{(a,a), (b,b), (c,a), (c,b), (c,c), (c,d), (c,e), (d,d), (e,d),(e,b),(e,e)\} $
 is an ordered semigroup which is not  regular.\\
 
  The egg box diagram  of $ S $ is as follows:\\

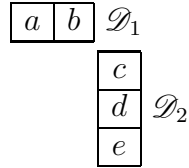
\begin{center}
	\begin{tabular}{|c|c|c|r|}
		\hline
		$ a $&  $ b$\\
		\hline
	\end{tabular} $\mathscr D_1$
\end{center}
\qquad\qquad\qquad\qquad\qquad\qquad\qquad\qquad	
	\begin{tabular}{|c|c|c|r|r|}
		\hline
		$c$\\
		\hline
			$d$\\
				\hline
			$e$\\
			\hline
	\end{tabular} $\mathscr D_2$

The bi-ideals of $ S $ are $ B_{1} = \{c\} $, $ B_{2} = \{c,e\} $, $ B_{3} = \{c,d,e\} $, $ B_{4} = \{a,c,d,e\}$, $ B_{5} = \{b,c,d,e\}$, $ B_{6} = S $.
The Cayley table of the semigroup of bi-ideals $\mathcal{B}(S)$   = $\{B_{1},B_{2}, B_{3}, B_{4}, B_{5}, B_{6}\}$ is given below:
\begin{center}
	\begin{tabular}{c|c|c|c|c|r|r}
		$*$& $ B_{1} $ & $ B_{2}$ & $ B_{3}$ & $ B_{4}$ & $ B_{5}$ & $ B_{6}$\\
		\hline
		$B_{1}$& $ B_{1} $ & $B_{1}$& $ B_{1} $ & $B_{1}$& $ B_{1} $& $ B_{1}$\\
		\hline
		$B_{2}$& $ B_{1} $ & $B_{1}$& $ B_{1} $ & $B_{1}$& $ B_{1} $& $ B_{1}$\\
		\hline
		$B_{3}$& $B_{3}$ & $B_{3}$ & $ B_{3}$ & $ B_{3}$ & $B_{3}$ & $B_{3}$\\
		\hline
		$B_{4}$& $B_{3}$& $B_{3}$ & $B_{3}$ & $ B_{4}$ & $B_{5}$& $B_{6}$ \\	
		\hline
		$B_{5}$& $B_{3}$& $B_{3}$ & $B_{3}$ & $ B_{4}$ & $B_{5}$& $B_{6}$ \\
		\hline
		$B_{6}$& $B_{3}$& $B_{3}$ & $B_{3}$ & $ B_{4}$ & $B_{5}$& $B_{6}$ \\	
	\end{tabular} 
\end{center}
The egg box diagram  of $\mathcal{B}(S)$ is given as follows:
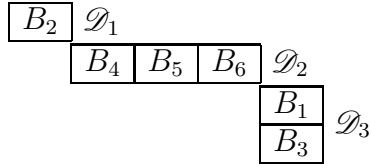
\begin{center}

	\begin{tabular}{|c|r|}
		\hline
		$ B_{2}$\\
		\hline
	\end{tabular} $\mathscr D_1$\\

\qquad\qquad\qquad\qquad
	\begin{tabular}{|c|c|c|r|}
		\hline
		$ B_{4} $&  $ B_{5} $ & $ B_{6} $\\
		\hline
	\end{tabular} $\mathscr D_2$\\

\qquad\qquad\qquad\qquad\qquad\qquad\qquad\qquad
		\begin{tabular}{|c|c|c|r|r|}
		\hline
		$ B_{1} $\\
		\hline
		$ B_{3} $\\
		\hline
			\end{tabular} $\mathscr D_3$
\end{center}
Clearly $ (B_{1}, B_{2}) \in \mathscr{L'} $ but $ (B_{1}, B_{2}) \notin \mathscr{L}_{\mathcal{B}(S)} $. Hence $ \mathscr{L'}  \nsubseteq \mathscr{L}_{\mathcal{B}(S)} $
\end{example}

\section{Green's relations on semigroup of bi-ideals of ordered full transformation semigroup}
In the following we describe the semigroup of bi-ideals of ordered full transformation semigroup and define the Green's relations. For, consider full transformation semigroup $ \mathcal{T}_{X} $ on a finite set $ X $  with natural partial ordering  given by
\begin{center}
	$ f \leq g \iff R(f) \subseteq R(g)$  and for some $ \alpha \in \mathcal{T}_{X,} \,\, f = \alpha f = \alpha g $
\end{center}
where $R(f)$ and $ R(g) $ are right ideal generated by $ f,\,\, g \in \mathcal{T}_{X} $ respectively. Then ($ \mathcal{T}_{X} $, $\leq$) is an regular ordered semigroup (see cf.\cite{kss}).
\begin{lemma}
Let	$ \mathcal{T}_{X} $ be the full tranformation semigroup  and 
	$(\mathcal{T}_{X} , \leq)$ be an ordered full transformation semigroup on a set $X$. Suppose that $ \mathscr{L} $ is the Green's relation on semigroups. For $ f, g \in  \mathcal{T}_{X}  $, then $ f \, \mathscr{L} \, g $ in ($ \mathcal{T}_{X} $, $\leq$) if and only if $ f \, \mathscr{L} \, g$ in $ \mathcal{T}_{X} $.
\end{lemma}
\begin{proof}
	Suppose that $ f \, \mathscr{L} \, g $ in ($ \mathcal{T}_{X} , \leq$). Then by Proposition \ref{1}, there exists $ x, y \in \mathcal{T}_{X} $ such that $ f \leq xg $ and $ g \leq yf $. By definition of natural partial order, $f \leq xg $ implies there exists $ \alpha \in \mathcal{T}_{X} $ such that $ f = \alpha xg $. Then  $Im f  \subseteq  Im g$. Similarly $ g \leq yf \implies  Im g  \subseteq  Im f$. Hence $ f \, \mathscr{L} \, g $ in ($ \mathcal{T}_{X} , \leq)$  implies  $Im f = Im g$ 
	and so $f \mathscr{L} \, g$ in $ \mathcal{T}_{X} $.\\
	Conversly suppose that $ f \mathscr{L} \, g$ in $ \mathcal{T}_{X} $, By equivalent condition of Green's $ \mathscr{L} $ relation $ \exists\,\, x, y \in \mathcal{T}_{X} $ such that $ f = xg $ and $ g = yf $. By the reflexivity of $ \leq $ we get $ f \leq xg $ and $ g \leq yf $. Hence $ f \, \mathscr{L} \, g $ in ($ \mathcal{T}_{X} , \leq$)
	\end{proof}
Similary we obtain the following lemma:
\begin{lemma}
	Let $ \mathcal{T}_{X} $ be the full tranformation semigroup on $X$ and $ (\mathcal{T}_{X} $, $\leq$) be the ordered full transformation semigroup. $ \mathscr{R} $ is the Green's $ \mathscr{R} $ relation on semigroups. For $ f, g \in  \mathcal{T}_{X} $, 
	$$ f \, \mathscr{R} \, g \in (\mathcal{T}_{X}, \leq) \iff f \, \mathscr{R} \, g \in 
	\mathcal{T}_{X}. $$
\end{lemma}

\begin{theorem}{\cite{mallick}}{\label{8}}
	Let $ S $ be a ordered regular semigroup and $  \mathcal{R}(S)$, $\mathcal{L}(S)$,  $\mathcal{B}(S)  $ be the collection of right, left, and bi-ideals of $ S $ respectively. Then $ \mathcal{R}(S)$ and $\mathcal{L}(S) $  are bands and $ \mathcal{B}(S) = \mathcal{R}(S)\mathcal{L}(S)$. 
\end{theorem}
\begin{theorem}{\cite{cao}}
	$ S $ is ordered regular semigroup if and only if for every right ideal $ R $ and left ideal $ L $ of $ S, (RL] = R \cap L$.
\end{theorem}
We have for any right ideal $ R $ and left ideal $ L $ of a regular ordered semigroup $ (RL] = R \cap L$. Also $ R * L = (RL] $ and any bi-ideal $  B = R * L $. Subject to these obervations we have the following theorem.
\begin{theorem} \label{3.4}
	Let $ S $ be a ordered regular  semigroup and $  \mathcal{R}(S)$, $\mathcal{L}(S)$,  $\mathcal{B}(S)  $ be the collection of right, left and bi-ideals of $ S $ respectively. Then any $ B \in \mathcal{B}(S) $ is of the form $ B =  R \cap L $, where $ R \in \mathcal{R}(S) $, $ L \in \mathcal{L}(S) $.
\end{theorem}
\begin{corollary}
	Let $ S $ be a ordered regular semigroup and $ L(a), R(a), B(a) $ be the principal left, right and bi-ideal generated by $ a \in{S} $  respectively. Then $ B(a) = R(a) \cap L(a)$.
\end{corollary}
\begin{proof}
	Let $ t \in B(a) \implies t \in (aSa] \implies t \leq axa, $ for some $ x \in S$. Then $ t \leq a(xa)  $ and $ t \leq (ax)a  \implies t \in R(a) $ and $t \in L(a) \implies t \in R(a) \cap L(a) \implies B(a) \subseteq R(a) \cap L(a) $.\\
	Now suppose that $ B(a) = R \cap L $, for some $ R \in \mathcal{R}(S) $, $ L \in \mathcal{L}(S) $. Then $ R(a) \subseteq R $ and $ L(a) \subseteq L \implies R(a) \cap L(a) \subseteq R \cap L \implies R(a) \cap L(a) \subseteq B(a).  $ Hence $ B(a) = R(a) \cap L(a). $
\end{proof}
In the light of Theorem \ref{3.4}  it can be seen that  any bi-ideal $ B $ of 
($ \mathcal{T}_{X} , \leq$) is of the form $ B = R * L = (RL] = R \cap L $, where $ R \in \mathcal{R}(\mathcal{T}_{X}) $, $ L \in \mathcal{L}(\mathcal{T}_{X}) $.

The following examples describes the semigroup of bi-ideals of ordered regular  semigroups 
$ (\mathcal{T}_{2},\leq))$ and $ (\mathcal{T}_{3},\leq))$ and it is shown that the relations $\mathscr{L'} $ and $\mathscr{R'} $ defined in Propositions 2.3 and 2.5 coinsides with Green's relations in the semigroups of bi-ideals of $(\mathcal{T}_{2},\leq)$ and $(\mathcal{T}_{3},\leq)$.

\begin{example} ( Bi-ideals of $ (\mathcal{T}_{2},\leq)).\,\,$ Let $ X = \{1, 2\} $,  denote $ (i, j) $ for the mapping $ 1 \mapsto i, 2 \mapsto j $. Then $ \mathcal{T}_{2} = \{(1,1),(2,2), (1,2), (2,1)\} $ and it's Cayley table and natural partial order $\leq$ are  given below:\\

 \begin{center}
 	\begin{tabular}{c|c|c|r|r}
 			$.$& $(1,1)$ & $(2,2)$ & $(1,2)$ & $(2,1)$\\
 					\hline
 		$(1,1)$& $(1,1)$ & $(2,2)$ & $(1,1)$ & $(2,2)$\\
\hline
$(2,2)$& $(1,1)$ & $(2,2)$ & $(2,2)$ & $(1,1)$\\ 
\hline
$(1,2)$& $(1,1)$ & $(2,2)$ & $(1,2)$ & $(2,1)$\\
\hline
$(2,1)$& $(1,1)$ & $(2,2)$ & $(2,1)$ & $(1,2)$\\		
 	\end{tabular} 
 \end{center}

 $ \leq =  \{(1,1) \leq (1,1), (2,2) \leq (2,2), (1,2) \leq (1,2), (2,1) \leq (2,1),\\
  (1,1) \leq (1,2), (1,1) \leq (2,1), (2,2) \leq (1,2), (2,2) \leq (2,1)\} $.\\
 Egg-box diagram of ($ \mathcal{T}_{2} , \leq$) is as follows:\\

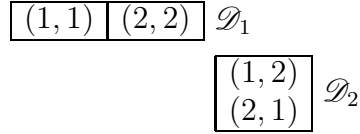
 \begin{center}
 	 \begin{tabular}{|c|c|c|r|}
 		\hline
 		$(1,1)$&  $(2,2)$\\
 		\hline
 	\end{tabular} $\mathscr D_1$
\end{center}
\qquad\qquad\qquad\qquad\qquad\qquad\qquad\qquad\qquad
	\begin{tabular}{|c|c|c|r|r|}
		\hline
	 $(1,2)$\\$ (2,1) $\\
	 		\hline
	\end{tabular} $\mathscr D_2$\\

Bi-ideals of ($ \mathcal{T}_{2} , \leq$) are $ B_{1} = \{(1, 1)\} $, $ B_{2} = \{(2, 2)\} $, $ B_{3} = \{(1, 1), (2, 2)\} $, $ B_{4} = \{(1, 1), (2, 2), (1,2), (2,1)\} $. The cayley table of $\mathcal{B}(\mathcal{T}_{2}) = \{B_{1}, B_{2}, B_{3}, B_{4}\}  $, the semigroup of bi-ideals of ($ \mathcal{T}_{2} , \leq$) is given below:\\

  \begin{center}
 	\begin{tabular}{c|c|c|r|r}
 		$*$& $ B_{1} $ & $ B_{2}$ & $ B_{3}$ & $ B_{4}$\\
 		\hline
 		$B_{1}$& $ B_{1} $ & $ B_{2}$ & $ B_{3}$ & $ B_{3}$\\
 		\hline
 		$B_{2}$& $ B_{1} $ & $ B_{2}$ & $ B_{3}$ & $ B_{3}$\\
 		\hline
 		$B_{3}$& $ B_{1} $ & $ B_{2}$ & $ B_{3}$ & $ B_{3}$\\
 		\hline
 		$B_{4}$& $ B_{1} $ & $ B_{2}$ & $ B_{3}$ & $ B_{4}$\\		
 	\end{tabular} 
 \end{center}

 Egg-box diagram of ($\mathcal{B}(\mathcal{T}_{2}), *$) is given as follows:\\
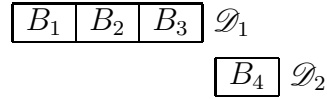
 \begin{center}
	\begin{tabular}{|c|c|c|r|}
		\hline
		$ B_{1} $ & $ B_{2}$ & $ B_{3}$\\
		
		\hline
	\end{tabular} $\mathscr D_1$
\end{center}
\qquad\qquad\qquad\qquad\qquad\qquad\qquad\qquad\qquad
	\begin{tabular}{|c|c|c|r|r|}
		\hline
		$ B_{4} $\\
		\hline
	\end{tabular} $\mathscr D_2$\\

The equivalence relations $\mathscr{L'} $ and $\mathscr{R'} $, defined in the Propositions \ref{a} and \ref{r} respectively are :\\
$\mathscr{L'}$ = $\{(B_{1}, B_{1}), (B_{2}, B_{2}), (B_{3}, B_{3}), (B_{4}, B_{4}) \}$\\
$\mathscr{R'} $ = $ \{(B_{1}, B_{1}), (B_{2}, B_{2}), (B_{3}, B_{3}), (B_{4}, B_{4}), (B_{1}, B_{2}), (B_{1}, B_{3}), (B_{2}, B_{3}),$ $ (B_{2}, B_{1}), (B_{3}, B_{1}), (B_{3}, B_{2})\} $.
\end{example}

Hence it is seen that the Green's relations  $\mathscr{L}_{\mathcal{B}(\mathcal{T}_{2})} $,  $\mathscr{R}_{\mathcal{B}(\mathcal{T}_{2})} $
on $ \mathcal{B}(\mathcal{T}_{2},\leq) $  coincides with $\mathscr{L'} $ and $\mathscr{R'} $ and the Green's relations  $\mathscr{L}_{\mathcal{B}(\mathcal{T}_{2})} $,  $\mathscr{R}_{\mathcal{B}(\mathcal{T}_{2})} $
on $ \mathcal{B}(\mathcal{T}_{2}) $ can be defined in terms of Green's relations  $\mathscr{L}_{\mathcal{T}_{2}} $,  $\mathscr{R}_{\mathcal{T}_{2}} $
on $ (\mathcal{T}_{2},\leq) $ .

From the Cayley table of $ \mathcal{B}(\mathcal{T}_{2}) $ it is evident that $ \mathcal{B}(\mathcal{T}_{2}) $ is a band. So by Theorem \ref{b}, ($ \mathcal{T}_{2} , \leq$) is both regular and intra- regular.\\

\begin{example} (bi-ideals of $(\mathcal{T}_{3},\;\leq)\,\,$

Let $ X = \{1, 2, 3\} $. Similar to in the case of $\mathcal{T}_{2}$, we shall denote $ (i, j, k) $ for the mapping $ 1 \mapsto i, 2 \mapsto j, 3 \mapsto k $ and order as the natural partial order. Then the egg box diagram of $\mathcal{T}_{3}$ is as below:\\

	\begin{tabular}{|c|c|c|r|}
		\hline
		$(1,1,1)$&  $(2,2,2)$ &  $(3, 3,3)$\\
		\hline
	\end{tabular} $\mathscr D_1$

\qquad\qquad\qquad\qquad\qquad\qquad
	\begin{tabular}{|c|c|c|r|r|}
		\hline
	$(1,2,2)$&  $(1,3,3)$ &  $(2,3,3)$\\
	$(2,1,1)$&  $(3,1,1)$ &  $(3,2,2)$\\
		\hline
			$(2,1,2)$&  $(3,1,3)$ &  $(3,2,3)$\\
			$(1,2,1)$&  $(1,3,1)$ &  $(2,3,2)$\\
		\hline
			$(2,2,1)$&  $(3,3,1)$ &  $(3,3,2)$\\
			$(1,1,2)$&  $(1,1,3)$ &  $(2,2,3)$\\
		\hline
	\end{tabular} $\mathscr D_2$\\

\qquad\qquad\qquad\qquad\qquad\qquad\qquad\qquad\qquad\qquad\qquad
\begin{tabular}{|c|c|c|r|r|}
		\hline
		$(1,2,3) (2,3,1) $\\
		$(3,1,2) (1,3,2) $\\
		$(3,2,1) (2,1,3) $\\
		\hline
	\end{tabular} $\mathscr D_3$

The right ideals of $\mathcal{T}_{3}$ are $ \mathcal{R}(\mathcal{T}_{3}) $, = \{  $\mathscr D_1$ , $\mathcal{T}_{3}$ and union of $\mathscr D_1$ and rows of  $\mathscr D_2$ \}
The principal left ideals are  $ L_{1} = \{(1,1,1)\} $ , $ L_{2} = \{(2,2,2)\} $ , $ L_{3} = \{(3,3,3)\} $ , $ L_{4} = \{(1,1,1),(2,2,2)\} \cup $\{first column of $\mathscr D_2$ \}  ,  $ L_{5} = \{(1,1,1),(3,3,3)\} \cup $\{second column of $\mathscr D_2$ \}  $ L_{6} = \{(2,2,2), (3,3,3)\} \cup $\{third column of $\mathscr D_2$ \}  and $\mathcal{T}_{3}$.
So the left ideals of $\mathcal{T}_{3}$ are  $ \mathcal{L}(\mathcal{T}_{3}) $ = \{$\mathcal{T}_{3}, L_{i}'s$ and all possible unions of $ L_{i}'s$\}.\\
Hence the bi-ideals of $\mathcal{T}_{3}$ are \\
$ \mathcal{B}(\mathcal{T}_{3}) $ = \{ $ B = R * L = (RL] = R \cap L $ | $ R \in \mathcal{R}(\mathcal{T}_{3}) $, $ L \in \mathcal{L}(\mathcal{T}_{3}) \}$.\\
Consider the equivalence relation $\mathscr{L'} $ on $ \mathcal{B}(\mathcal{T}_{3})$. i.e. $ B_{i} \, \mathscr{L'} \, B_{j} $ if and only if for each $ b_{i} \in{B_{i}} $ there exists $ b_{j} \in{B_{j}} $ such that $ b_{i}\, \mathscr{L}_{\mathcal{T}_{3}} \, b_{j}  $ and for each $ b'_{j} \in{B_{j}} $ there exists $ b'_{i} \in{B_{i}} $ such that $ b'_{i}\, \mathscr{L}_{\mathcal{T}_{3}} \, b'_{j}$. We can show that $ \mathscr{L'}  \subseteq \mathscr{L}_{\mathcal{B}(\mathcal{T}_{3})} $. For let $( B_{i}, B_{j}) \in \mathscr{L'} $. Since each bi-ideal of $ \mathcal{T}_{3} $ is the disjoint union of $ \mathscr{H} $ - classes, choose $ b_{i1} $ form one of those $ \mathscr{H} $- class of $ B_{i} $ . $ b_{i1} \in{B_{i}} $ and $ B_{i} \, \mathscr{L'} \, B_{j} \implies \exists  b_{j1} \in{B_{j}}  $ such that $ b_{i1}\, \mathscr{L}_{\mathcal{T}_{3}} \, b_{j1}  $. By proposition \ref{1} $ \exists \,x,y \in \mathcal{T}_{3} $ such that $ b_{i1} \leq xb_{j1} $ and $ b_{j1} \leq yb_{i1} $. In particular we can find $ x_{1} $ from $ \mathscr{D} $ - class of $ b_{i1} $ such that $ b_{i1} \leq x_{1}b_{j1} $. similarly choose choose $ b_{i2} $ form another  $ \mathscr{H} $- class of $ B_{i} $ and find $ x_{2} $ from $ \mathscr{D} $ - class of $ b_{i2} $ such that $ b_{i2} \leq x_{2}b_{j2} $ and so on. Then we can easily verify that the union of those $ \mathscr{H} $- class of $ x_{ik} $ is a bi-ideal of $ \mathcal{T}_{3} $. i.e. $ X = \cup \mathscr{H}(x_{ik}) \in \mathcal{B}(\mathcal{T}_{3}) $. Also $ B_{i} = X * B_{j} $. Similarly we get $ B_{j} = Y * B_{i} $. Hence $( B_{i}, B_{j}) \in \mathscr{L}_{\mathcal{B}(\mathcal{T}_{3})} $. So $ \mathscr{L'}  \subseteq \mathscr{L}_{\mathcal{B}(\mathcal{T}_{3})} $. We already proved in the Proposition \ref{2} that  $ \mathscr{L}_{\mathcal{B}(\mathcal{T}_{3})}  \subseteq  \mathscr{L'} $. Hence 
$\mathscr{L}_{\mathcal{B}(\mathcal{T}_{3})}  =  \mathscr{L'} $, \, 
i.e the Green's relation  $\mathscr{L}_{\mathcal{B}(\mathcal{T}_{3})} $,  on $\mathcal{B}(\mathcal{T}_{3}) $ is coincides with $\mathscr{L'} $ . Similarly $\mathscr{R}_{\mathcal{B}(\mathcal{T}_{3})} $ is same as  $\mathscr{R'} $. Hence the Green's relations  $\mathscr{L}_{\mathcal{B}(\mathcal{T}_{3})} $,  $\mathscr{R}_{\mathcal{B}(\mathcal{T}_{3})} $ on $ \mathcal{B}(\mathcal{T}_{3}) $ can be defined in terms of Green's relations  $\mathscr{L}_{\mathcal{T}_{3}} $,  $\mathscr{R}_{\mathcal{T}_{3}} $ on $ \mathcal{T}_{3} $.
\end{example}
We can extend this result to semigroup of bi-ideals of ordered full transformation semigroup $ \mathcal{T}_{X}$  with $ |X| = n $. The Green's relations on $ \mathcal{B}(\mathcal{T}_{X})$ is defined as follows:
\begin{definition}\label{def1}
	Given $ B_{1}, B_{2} \in \mathcal{B}(\mathcal{T}_{X}) $, $ B_{1}\,  \mathscr{L}_{\mathcal{B}(\mathcal{T}_{X})} \, B_{2} $ if and only if for each $ b_{1} \in{B_{1}} $ there exists $ b_{2} \in{B_{2}} $ such that $ b_{1}\, \mathscr{L}_{\mathcal{T}_{X}} \, b_{2}  $ and for each $ b'_{1} \in{B_{1}} $ there exists $ b'_{2} \in{B_{2}} $ such that $ b'_{1}\, \mathscr{L}_{\mathcal{T}_{X}} \, b'_{2}  $ .
\end{definition}

\begin{definition}
	Given $ B_{1}, B_{2} \in \mathcal{B}(\mathcal{T}_{X}) $, $ B_{1}\,  \mathscr{R}_{\mathcal{B}(\mathcal{T}_{X})} \, B_{2} $ if and only if for each $ b_{1} \in{B_{1}} $ there exists $ b_{2} \in{B_{2}} $ such that $ b_{1}\, \mathscr{R}_{\mathcal{T}_{X}} \, b_{2}  $ and for each $ b'_{1} \in{B_{1}} $ there exists $ b'_{2} \in{B_{2}} $ such that $ b'_{1}\, \mathscr{R}_{\mathcal{T}_{X}} \, b'_{2}  $ .
\end{definition}
\begin{proposition}
	Let $ f, g \in  $ $ (\mathcal{T}_{X} $, $\leq$) and $ B(f), B(g) $ be the principal bi-ideal generated by $ f, g $ respectively. Then $ f\, \mathscr{L}_{\mathcal{T}_{X}} \, g  $ if and only if $ B(f)\,  \mathscr{L}_{\mathcal{B}(\mathcal{T}_{X})} \, B(g) $.  
\end{proposition}
\begin{proof}
	Suppose that $ f\, \mathscr{L}_{\mathcal{T}_{X}} \, g  $. Then $ Im f  = Im g $. Let $ t \in B(f)$. Then $ t \leq fhf $ for some $ h \in $ $ (\mathcal{T}_{X} $, $\leq$). By the definition of natural partial order $ \exists \alpha \in $ $ (\mathcal{T}_{X} $, $\leq$) such that $ t = \alpha fhf \implies Im t \subseteq Im f \implies Im t \subseteq Im g $. We have to find $ t' \in B(g) $ such that $ Im t = Im t' $. Let $ Im g = \{x_{1}, x_{2},...x_{l}\} \subseteq X $ and $ Im t = \{x_{i1}, x_{i2},...x_{ik}\} \subseteq Im g $. Define a map $ y : X \rightarrow X $ as follows: Each $ x_{ij} \mapsto y_{ij} $, where $ y_{ij} $ is an element of set of pre images of $ x_{ij},j = 1, 2,...k $ under the map $ g $ and $ X - \{x_{i1}, x_{i2},...x_{ik}\} \mapsto y_{i1} $. Now define $ t' = gyg $. Then $ t' \in B(g) $ and $ Im t = Im t' \implies t\, \mathscr{L}_{\mathcal{T}_{X}} \, t'  $. Similarly we can prove that for each $ q \in B(g),$ there exists $ q' \in B(f) $ such $ q\, \mathscr{L}_{\mathcal{T}_{X}} \, q' $. So by Definition \ref{def1}, $ B(f)\,  \mathscr{L}_{\mathcal{B}(\mathcal{T}_{X})} \, B(g) $.\\
	Conversely suppose that  $ B(f)\,  \mathscr{L}_{\mathcal{B}(\mathcal{T}_{X})} \, B(g) $. By equivalent condition of Green's relation in semigroup, $ \exists \,X,Y \in \mathcal{B}(\mathcal{T}_{X}) $ such that $ B(f) = X * B(g)  $ and $ B(g) = Y * B(f)  \implies B(f) = (X B(g)] $ and $ B(g) = (Y B(f)] $ . Then $ f \leq xt $ and $ g \leq yt' $, where $ x \in X, y \in Y, t \in B(g), t' \in B(f)$. $ t \in B(g) \implies t \leq gqg $ and $ t' \in B(f) \implies t' \leq fq'f, \, (q, q'\in$  $ (\mathcal{T}_{X} $, $\leq$)). So $ f \leq xgqg $ and $ g \leq yfq'f \implies f \leq (xgq)g $ and $ g \leq (yfq')f $. Hence by Proposition \ref{1},  $ f\, \mathscr{L}_{\mathcal{T}_{X}} \, g  $.
	 
\end{proof}
Consider the semigroup $ (\mathcal{B}(\mathcal{T}_{X}), *) $. Let $ B \in \mathcal{B}(\mathcal{T}_{X}) $. Since every right and left ideals are idempotents in $  \mathcal{B}(\mathcal{T}_{X})(Theorem \ref{8}) $,  if $ B $ is left or right ideal then $ \mathscr{L}_{\mathcal{B}(\mathcal{T}_{X})} $ - class and $ \mathscr{R}_{\mathcal{B}(\mathcal{T}_{X})} $ - class of $ B $ contains the idempotent $ B $. If $ B $ is neither a left ideal nor a right ideal then $ B $ is of the form $ B = R \cap L $,  where $ R \in \mathcal{R}(\mathcal{T}_{X}) $, $ L \in \mathcal{L}(\mathcal{T}_{X}) $. It is evident that $ B \,  \mathscr{L}_{\mathcal{B}(\mathcal{T}_{X})} \, L  $ and $ B\,  \mathscr{R}_{\mathcal{B}(\mathcal{T}_{X})} \, R $ . Hence $ \mathscr{L}_{\mathcal{B}(\mathcal{T}_{X})} $ - class of B  contains the idempotent $ L $
and $ \mathscr{R}_{\mathcal{B}(\mathcal{T}_{X})} $ - class of $ B $ contains the idempotent $ R $. So $ (\mathcal{B}(\mathcal{T}_{X}), *) $ is a regular semigroup.

	\end{document}